\documentclass[10pt]{amsart}

\usepackage{amsmath,amsfonts, latexsym,graphicx, amssymb, mathtools, enumerate, mdwlist, amscd}

\usepackage{hyperref}
\hypersetup{colorlinks=true, urlcolor=blue, citecolor=blue, linkcolor=blue}

\newcommand{\U}{{\mathcal U}}
\newcommand{\0}{{\mathbf 0}}
\newcommand{\C}{{\mathbb C}}
\newcommand{\Z}{{\mathbb Z}}

\newcommand{\Q}{{\mathbb Q}}

\newcommand{\Pdot}{\mathbf P^\bullet}

\newcommand{\coker}{{\operatorname{coker}}}

\newcommand{\mfp}{\mathfrak p}
\newcommand{\mult}{{\operatorname{mult}}}

\newtheorem{defn0}{Definition}[section]
\newtheorem{prop0}[defn0]{Proposition}
\newtheorem{conj0}[defn0]{Conjecture}
\newtheorem{thm0}[defn0]{Theorem}
\newtheorem{lem0}[defn0]{Lemma}
\newtheorem{corollary0}[defn0]{Corollary}
\newtheorem{example0}[defn0]{Example}
\newtheorem{remark0}[defn0]{Remark}
\newtheorem{question0}[defn0]{Question}
\newtheorem{exercise0}[defn0]{Exercise}

\newenvironment{prop}{\begin{prop0}}{\end{prop0}}

\newenvironment{thm}{\begin{thm0}}{\end{thm0}}

\newenvironment{exm}{\begin{example0}\rm}{\end{example0}}
\newenvironment{rem}{\begin{remark0}\rm}{\end{remark0}}

\newcommand{\propref}[1]{Proposition~\ref{#1}}
\newcommand{\thmref}[1]{Theorem~\ref{#1}}

\newcommand{\mbf}[1]{{\mathbf #1}}

\title{L\^e modules and hypersurfaces with one-dimensional singular sets}

\dedicatory{\hfill in memory of my mentor and friend, L\^e D\~ung Tr\'ang}


\author{David B. Massey}

\date{}

\begin{document}

\begin{abstract} By using our previous results on L\^e modules and an upper-bound on the betti numbers which we proved with L\^e, we investigate the cohomology of Milnor fibers and the internal local systems given by the vanishing cycles of hypersurfaces with one-dimensional singular sets and small L\^e numbers.
\end{abstract}

\maketitle




\section{Introduction}

 \medskip
 
 Let $\U$ be a connected open neighborhood of the origin in $\C^{n+1}$, let $\mbf z:=(z_0, \dots, z_n)$ be coordinates on $\U$,  and let $f:(\U, \0)\rightarrow (\C, 0)$ be a non-constant, reduced complex analytic function. After possibly choosing $\U$ smaller, the critical locus, $\Sigma f$, of $f$ is equal to the singular set $\Sigma V(f)$ of the hypersurface $V(f):=f^{-1}(0)$. In this paper, we assume that $\dim_\0\Sigma f=1$ and  assume, possibly after an analytic change of coordinates, that $z_0$ is generic enough so that $\dim_\0\Sigma (f_{|_{V(z_0)}})=0$ and so that, for all irreducible components $C$ of $\Sigma f$ at $\0$, $\left(C\cdot V(z_0)\right)_\0={\operatorname{mult}}_\0C$.

\medskip
 
 In many places, we have discussed the L\^e numbers of hypersurface singularities; see, for instance, \cite{levar1}, \cite{levar2},  \cite{lecycles}, \cite{lemassey}, and \cite{lemodtrace}.  In our current setting, there are only two possibly non-zero L\^e numbers at the origin, $\lambda^0_{f, \mbf z}(\0)$ and $\lambda^1_{f, \mbf z}(\0)$, which we will henceforth denote simply as  $\lambda^0$ and $\lambda^1$; these are guaranteed to exist, since our genericity assumption on $z_0$ implies the coordinates $\mbf z$ are {\it prepolar} for $f$ at $\0$. In the next section, we will recall the definitions of $\lambda^0$ and $\lambda^1$, but here let us just say that $\lambda^0$ and $\lambda^1$ are effectively calculable, by hand in many examples and by computer in harder examples. 
 
 For specific low values of $\lambda^0$ and $\lambda^1$, we want to look at results on the reduced cohomology $\widetilde H^*()$ of the Milnor fiber $F_{f, \0}$, which can possibly be non-zero only in degrees $(n-1)$ and $n$.
 
 There will be subcases in which we also need to consider 
 $$
 r:= \textnormal{ the number of irreducible components of } \Sigma f \textnormal{ at }\0 \ \textnormal{ and }  \ m:={\operatorname{mult}}_\0|\Sigma f|,
 $$
 where the vertical lines indicate that we are considering $\Sigma f$ with its reduced structure, i.e., as a set.
 
 \medskip
 
 We give our results here in the introduction, but will defer the proofs until Section 3.
 
 \vskip 0.2in
 
 As we are assuming that $\dim_\0\Sigma f=1$, $\lambda^1\neq 0$, so let us start with the case:
 
 \smallskip
 
 \begin{prop}\label{prop:0=0} Suppose that $\lambda^0=0$.
 
 \smallskip
 
 Then, at $\0$, $\Sigma f$ is smooth, transversely intersected by $V(z_0)$, and  $f_{z_0}(z_1, \dots, z_n)=f(z_0, \dots, z_n)$ defines a 1-parameter family of isolated singularities along $C$ with constant Milnor number $\mu^\circ$.
 
Furthermore,
$$\widetilde H^{n-1}(F_{f, \mathbf 0};\ \Z)\cong \Z^{\mu^\circ}\hskip 0.2in\textnormal{ and }\hskip 0.2in\widetilde H^{n}(F_{f, \mathbf 0};\ \Z)=0.$$
 \end{prop}
 
 \medskip

 \begin{prop}\label{prop:1=1} Suppose $\lambda^1=1$ and $\lambda^0\geq 1$.
 
 \smallskip
 
 Then,  
 $$\widetilde H^{n-1}(F_{f, \mathbf 0};\ \Z)=0\hskip 0.2in\textnormal{ and }\hskip 0.2in \widetilde H^{n}(F_{f, \mathbf 0};\ \Z)\cong \Z^{\lambda^0-1},$$
 and $\lambda^0\geq 2$.
 
 \end{prop}
 
 \medskip
 
 \begin{prop}\label{prop:0=1} Suppose $\lambda^0=1$ and $\lambda^1\geq 2$.
 
 \smallskip
 
Then $\mult_\0 f=2$, $m=2$, and
$$\widetilde H^{n-1}(F_{f, \mathbf 0};\ \Z)\cong\Z^{\lambda^1-1}\hskip 0.2in\textnormal{ and }\hskip 0.2in\widetilde H^{n}(F_{f, \mathbf 0};\ \Z)=0.$$ 

\smallskip

Furthermore,  if $r=1$, then $\lambda^1=2$, and so $\widetilde H^{n-1}(F_{f, \mathbf 0};\ \Z)\cong\Z$.
\end{prop}

\medskip

\begin{prop}\label{prop:1=2} Suppose $\lambda^1=2$ and $\lambda^0\geq 2$.
 
 \smallskip

Then, in fact, $\lambda^0\geq 3$ and there are two cases:

\smallskip

\noindent Case 1: 
$$\widetilde H^{n-1}(F_{f, \mathbf 0};\ \Z)=0\hskip 0.2in\textnormal{ and }\hskip 0.2in \widetilde H^{n}(F_{f, \mathbf 0};\ \Z)\cong \Z^{\lambda^0-2}\oplus \big(\Z/b\Z\big)$$  for some $b\geq 1$.

\bigskip

\noindent Case 2:
$$\widetilde H^{n-1}(F_{f, \mathbf 0};\ \Z)\cong \Z\hskip 0.2in\textnormal{ and }\hskip 0.2in\widetilde H^{n}(F_{f, \mathbf 0};\ \Z)\cong \Z^{\lambda^0-1}.$$

\medskip

Furthermore, if we also assume that $m=1$, then we are in Case (1) above and
$$\widetilde H^{n-1}(F_{f, \mathbf 0};\ \Z)=0\hskip 0.2in\textnormal{ and }\hskip 0.2in \widetilde H^{n}(F_{f, \mathbf 0};\ \Z)\cong \Z^{\lambda^0-2}\oplus \big(\Z/b\Z\big),$$ where either $b=1$ (so there is no torsion) or $b$ is a product of prime powers where the primes are either $3$ or congruent to $1$ modulo $6$. 
\end{prop}

\medskip

\begin{prop}\label{prop:0=2} Suppose $\lambda^0=2$ and $\lambda^1\geq 3$.
 
 \smallskip

There are two cases:

\smallskip

\noindent (1) 
$$\widetilde H^{n-1}(F_{f, \mathbf 0};\ \Z)\cong\Z^{\lambda^1-2}\hskip 0.2in\textnormal{ and }\hskip 0.2in \widetilde H^{n}(F_{f, \mathbf 0};\ \Z)\cong \Z/b\Z$$  for some $b\geq 1$.

\bigskip

\noindent (2) 
$$\widetilde H^{n-1}(F_{f, \mathbf 0};\ \Z)\cong \Z^{\lambda^1-1}\hskip 0.2in\textnormal{ and }\hskip 0.2in\widetilde H^{n}(F_{f, \mathbf 0};\ \Z)\cong \Z.$$

\bigskip

Furthermore, if $m=2$, then we are in Case (1) above and either $b=1$ or $b$ is a product of prime powers where the primes are either $3$ or congruent to $1$ modulo $6$. 
\end{prop}

\medskip

\medskip

\begin{prop}\label{prop:equal} Suppose $\lambda^0=\lambda^1= 3$.

\smallskip

Then
$$\widetilde H^{n-1}(F_{f, \mathbf 0};\ \Z)\cong \Z^2\hskip 0.2in\textnormal{ and }\hskip 0.2in\widetilde H^{n}(F_{f, \mathbf 0};\ \Z)\cong \Z^2.$$

\smallskip

Furthermore, if $r=1$, then $m=1$, i.e.,  $\Sigma f$ is smooth and transversely intersected by $V(z_0)$ at $\0$
\end{prop}
 
 \smallskip

\vskip 0.2in

This paper is organized as follows. In Section 2, we will give definitions and prior results. Section 3 is devoted to proving the propositions/cases given above. Section 4 contains examples and mentions when we have none. Section 5 contains concluding remarks.

\medskip

\section{Definitions and prior results}

\smallskip
 
We should be clear from the beginning that, when we look at intersection numbers,  we deal with analytic cycles, {\bf not} cycle classes, and our intersection theory is the simple theory using proper intersections inside the complex manifold $\U$; one can find this in \cite{fulton} in Section 1.5, Section 8.2, and Example 11.4.4.

\medskip

Our assumption that 
$$\dim_\0\Sigma\big(f_{|_{V(z_0)}}\big) \ = \ \dim_\0V\left(z_0, \frac{\partial f}{\partial z_1}, \dots, \frac{\partial f}{\partial z_n}\right) \ = \ 0$$ is precisely equivalent to saying that $V\left(\frac{\partial f}{\partial z_1}, \dots, \frac{\partial f}{\partial z_n}\right)$ is purely $1$-dimensional (and not empty) at the origin and is properly intersected at the origin by $V(z_0)$.

In terms of analytic cycles, 
$$
V\left(\frac{\partial f}{\partial z_1}, \dots, \frac{\partial f}{\partial z_n}\right) \ = \ \Gamma^1_{f,\mbf z} \ + \ \Lambda^1_{f, \mbf z},
$$
where $\Gamma^1_{f,\mbf z}$ is the (possibly non-reduced) relative polar curve of $f$ of Hamm, L\^e, and Teissier, which consists of components not contained in $\Sigma f$, and $\Lambda^1_{f, \mbf z}$ is the $1$-dimensional L\^e cycle, which consists of components which are contained in $\Sigma f$. See Definition 1.11 of \cite{lecycles}. 
Note that $V\left(\frac{\partial f}{\partial z_0}\right)$ necessarily intersects $\Gamma^1_{f,\mbf z}$ properly at $\0$, and that $V(z_0)$ intersects $\Lambda^1_{f, \mbf z}$ properly at $\0$ by our assumption.

Letting $C$'s denote the underlying reduced components of $\Sigma f$ at $\0$, at the origin, we have
$$
\Lambda^1_{f, \mbf z}  \ = \ \sum_C\mu^\circ_C \,C,
$$
where $\mu^\circ_C$ is the Milnor number of $f$ restricted to a generic hyperplane slice, at a point $\mbf p$ on $C-\{\0\}$ close to $\0$. See Remark 1.19  of \cite{lecycles}.

The intersection numbers $\left(\Gamma^1_{f, \mbf z} \cdot V\left(\frac{\partial f}{\partial z_0}\right)\right)_\0$ and $\left(\Lambda^1_{f, \mbf z} \cdot V(z_0)\right)_\0$ are, respectively, the {\bf L\^e numbers $\lambda^0=\lambda^0_{f,\mbf z}$ and $\lambda^1=\lambda^1_{f,\mbf z}$ (at the origin)}. See Definition 1.11 of \cite{lecycles}. Note that
$$\lambda^1=\sum_C \left(C\cdot V(z_0)\right)_\0\mu^\circ_C.$$

\bigskip

We wish now to discuss the main result of \cite{lemassey}. Note that $\widetilde H^{n-1}(F_{f,\0};\Z)$ is free abelian, while $\widetilde H^{n}(F_{f,\0};\Z)$ may have torsion. We write
$$
\widetilde H^{n-1}(F_{f,\0};\Z)\cong\Z^{\tilde b_{n-1}} \ \textnormal{ and }\ \widetilde H^{n}(F_{f,\0};\Z)\cong\Z^{\tilde b_{n}}\oplus T,
$$
where $T$ consists of pure torsion. For each prime number $\mfp$, we write $\tau_\mfp$ for the number of cyclic summands of $T$ with $\mfp$-torsion, i.e., the number of direct summands of $T$ of the form $\Z/\mfp^k\Z$.

The Universal Coefficient Theorem tells us that
$$\widetilde H^{n-1}(F_{f,\0}; \Z/\mfp\Z)\cong \left(\Z/\mfp\Z\right)^{\tilde b_{n-1}+\,\tau_\mfp} \ \textnormal{ and }\  \widetilde H^{n}(F_{f,\0};\Z/\mfp\Z)\cong \left(\Z/\mfp\Z\right)^{\tilde b_{n}+\,\tau_\mfp}.$$

\medskip

Now Theorem 2.5 of  \cite{lemassey} tells us:

\begin{thm}\label{thm:bound} 

\phantom{stuff}

\medskip

\noindent Either 

\medskip

\noindent $\bullet$ \ $\lambda^0=0$, $\Sigma f$ is smooth at $\0$ and $f$ defines a 1-parameter family of isolated singularities with constant Milnor number $\mu$, in which case $\lambda^0=0$,  $\widetilde H^{n-1}(F_{f, \mathbf 0};\ \Z)\cong \Z^{\mu}$, and $\widetilde H^{n}(F_{f, \mathbf 0};\ \Z)=0$,

\medskip

\noindent or

\medskip

\noindent $\bullet$ \ $\lambda^0\neq 0$ and, for all primes $\mfp$, there is a strict inequality  \ $\tilde b_{n-1}+\,\tau_\mfp<\lambda^1$ and, equivalently, $\tilde b_{n}+\,\tau_\mfp<\lambda^0$.
\end{thm}

\medskip

\vskip 0.15in

We will now define the L\^e {\bf modules}.  We fix a base ring 
$R$ which is a principal ideal domain (e.g., $\Z$ or any field), and consider the shifted constant sheaf (complex) $R_{\U}^\bullet[n+1]$; this is a {\it perverse} sheaf (see, for instance,  \cite{BBD}, \cite{kashsch},  \cite{dimcasheaves}).  It follows from the general theory that the shifted vanishing cycle complex along $f$, $\Pdot:=\phi_f[-1]R_{\U}^\bullet[n+1]$, is also a perverse sheaf, as are $\psi_{z_0}[-1]\Pdot$ and $\phi_{z_0}[-1]\Pdot$, the shifted nearby and vanishing cycles of $\Pdot$ along $z_0$. Furthermore the origin is an isolated point in the support of $\psi_{z_0}[-1]\Pdot$ and $\phi_{z_0}[-1]\Pdot$, and there is a canonical morphism 
$$\psi_{z_0}[-1]\Pdot\xrightarrow{{ \ \operatorname{can}}_{z_0} \ }\phi_{z_0}[-1]\Pdot.
$$

\medskip

 \noindent As discussed in \cite{lemodtrace}, the stalk cohomologies in degree $0$ of $\psi_{z_0}[-1]\Pdot$ and $\phi_{z_0}[-1]\Pdot$ (the {\bf L\^e modules}) are free modules of ranks $\lambda^1$ and $\lambda^0$, respectively, i.e.,
 $$
 H^0(\psi_{z_0}[-1]\Pdot)_\0\cong R^{\lambda^1}\hskip 0.2in\textnormal{ and }\hskip 0.2in  H^0(\phi_{z_0}[-1]\Pdot)_\0\cong R^{\lambda^0}
 $$
 and the map ${\operatorname{can}}_{z_0}$ induces the map $\partial$ from the introduction, i.e., an
 $R$-module homomorphism $\partial: R^{\lambda^1}\rightarrow R^{\lambda^{0}}$
 such that
 $$
 \ker\partial \cong \widetilde H^{n-1}(F_{f, \mathbf 0};\ R)\hskip 0.2in\textnormal{ and }\hskip 0.2in \coker\,\partial \cong \widetilde H^{n}(F_{f, \mathbf 0};\ R).
 $$
 
 \medskip
 
 In the context of a singular set of dimension 1, and recalling that $m:={\operatorname{mult}}_\0|\Sigma f|$, what we prove  in  \cite{lemodtrace} is that 
 \begin{thm}\label{thm:lemod}
 
 \phantom{stuff}
 
 \begin{itemize}
 \item There exists a $\Z$-module homomorphism $\partial: \Z^{\lambda^1}\rightarrow\Z^{\lambda^{0}}$
 such that
 $$
 \ker\partial \cong \widetilde H^{n-1}(F_{f, \mathbf 0};\ \Z)\hskip 0.2in\textnormal{ and } \coker\,\partial \cong \widetilde H^{n}(F_{f, \mathbf 0};\ \Z),
 $$
 where $\widetilde H^\bullet$ denotes reduced cohomology.
 
 \medskip
 
 \item For  $j=0,1$, the Milnor monodromy of $f$ induces a {\it L\^e-Milnor} monodromy automorphism $\alpha _j$ from $\Z^{\lambda^{j}}$ to itself which is compatible with the map $\partial$, i.e., $\partial\circ\alpha_1= \alpha_{0}\circ\partial$.
 
 \medskip
 
 \item All of the eigenvalues of each $\alpha_j$ are roots of unity, and so the characteristic polynomial of each $\alpha_j$ factors into cyclotomic polynomials.
 
 \medskip
 
 \item The traces, $\operatorname{tr} \alpha_j$, of $\alpha_0$ and $\alpha_1$ are given by
 $$
 \operatorname{tr} \alpha_0= (-1)^n(m-1)\hskip 0.2in\textnormal{ and }\hskip 0.2in \operatorname{tr} \alpha_1= (-1)^{n}\,m.
 $$

As all of the eigenvalues of $\alpha_0$ are roots of unity, it follows that  $m-1\leq \lambda^0$, with equality holding if and only if the characteristic polynomial of $\alpha_0$ is either $(t-1)^{\lambda^0}$ or $(t+1)^{\lambda^0}$.

 \medskip

 \item if $\mfp$ is a prime number, then tensoring $\partial$ with $\Z/\mfp\Z$ yields a linear transformation of  $\Z/\mfp\Z$-vector spaces  $ \partial^{\mfp}: \left(\Z/\mfp\Z\right)^{\lambda^1}\rightarrow \left(\Z/\mfp\Z\right)^{\lambda^{0}}$
  such that
 $$
 \ker\partial^{\mfp}\cong \widetilde H^{n-1}(F_{f, \mathbf 0};\ \Z/\mfp\Z)\hskip 0.2in\textnormal{ and } \coker\,\partial^{{}\,\mfp} \cong \widetilde H^{n}(F_{f, \mathbf 0};\ \Z/\mfp\Z).
 $$
 \end{itemize}
 \end{thm}

 \vskip 0.2in
 
The perversity of $\Pdot$ implies, among other things, that, for $\U$ small enough, for all components $C$ of $\Sigma f$, the cohomology sheaf $\mbf H^{-1}\big(\Pdot_{|_{C-\{\0\}}}\big)$ is a local system on $C-\{\0\}$ with stalk isomorphic to $\Z^{\mu^\circ_C}$.

 In classical terms, this is well-known; it says that, for $p\in C-\{\0\}$, (i) \ the reduced cohomology $\widetilde H^*(F_{f, p})$  is possibly non-zero only in degree $n-1$, (ii) \  the isomorphism-type of the module $\widetilde H^{n-1}(F_{f, p})$ is independent of $p\in C-\{\0\}$ and is free abelian of rank that we denoted above by $\mu^\circ_C$, and (iii) \ there is a monodromy action/homomorphism $\rho_C:\pi_1(C-\{\0\})\cong \Z\rightarrow \operatorname{Aut}(\Z^{\mu^\circ})$. Let $h_C$ denote the automorphism $\rho_C(1)$.

This internal local system monodromy is sometimes referred to as the {\it vertical monodromy}; this is to distinguish it from the Milnor monodromy induced by letting the value of $f$ travel around a small circle.

 Restricting our setting to a small neighborhood of the origin, the cosupport  condition for perverse sheaves (see also Siersma  \cite{siersmavarlad}) implies:
 
 \begin{prop}\label{prop:inj}
 $H^{-1}(\Pdot)_\0\cong\widetilde H^{n-1}(F_{f, \0}; \Z)$ injects  into $\bigoplus _C\ker\{\operatorname{id}-h_C\}$. In particular, $\tilde b_{n-1}\leq \sum_C \mu^\circ_C$.
\end{prop}

\bigskip

We have one other result to mention here.

\begin{prop}\label{prop:equal2} Suppose that $\tilde b_{n-1}$ equals 0 or 1. Then $\lambda^0\neq\lambda^1$.
\end{prop}
\begin{proof} Suppose that $\tilde b_{n-1}=0$ and $\lambda^0=\lambda^1$. Then we would also have $\tilde b_{n}=0$, which would contradict A'Campo's Theorem in \cite{acamp} that the Lefschetz number of the monodromy on the Milnor fiber is 0.

Suppose that $\tilde b_{n-1}=1$ and $\lambda^0=\lambda^1$. Then we would also have $\tilde b_{n}=1$, which would  imply that the traces of the Milnor automorphisms on $\widetilde H^{n-1}(F_{f, \0}; \Q)$ and $\widetilde H^{n}(F_{f, \0}; \Q)$ are each $\pm1$. This would again contradict A'Campo's Theorem.
\end{proof}

\medskip

Now we are prepared to prove all of the propositions/cases from the introduction.

\medskip

\section{Proofs}

\medskip

\noindent {\bf Proof of \propref{prop:0=0}.} Suppose that $\lambda^0=0$.

\smallskip

$\lambda^0=0$ if and only if the relative polar curve near the origin is empty as a set, or is 0 as a cycle; it is well-known (see, for instance,  Theorem 5.1 of \cite{lemassey}) that this is equivalent to: 

\bigskip

\noindent $\Sigma f$ is smooth at $\0$, and so $\Lambda^1=\mu^\circ C$, for some positive integer $\mu^\circ$ and some irreducible curve $C$ which is smooth at $\0$, $V(z_0)$ transversely intersects $C$ at $\0$, and  $f_{z_0}(z_1, \dots, z_n)=f(z_0, \dots, z_n)$ defines a 1-parameter family of isolated singularities along $C$ with constant Milnor number $\mu^\circ$.

\bigskip

Moreover, these equivalent conditions imply that $\widetilde H^{n-1}(F_{f, \mathbf 0};\ \Z)\cong \Z^{\mu^\circ}$ and $\widetilde H^{n}(F_{f, \mathbf 0};\ \Z)=0$. $\qed$

\bigskip

\noindent {\bf Proof of \propref{prop:1=1}.} Suppose $\lambda^1=1$ and $\lambda^0\geq 1$.

\smallskip
 In this case, $\Lambda^1_{f, \mbf z}=C$, where $C$ is smooth at $\0$, $\mu^\circ_C=1$, and $V(z_0)$ transversely intersects $C$ at $\0$. After an analytic change of coordinates, we may assume that $C$ is a complex line, and we are in the case that Siersma refers to as an {\it isolated line singularity} in \cite{siersmaisoline}. As we are assuming that $\lambda^0\neq 0$, Siersma proves in this case that $F_{f, \0}$ has the homotopy-type of a bouquet of $n$-spheres; the number of $n$-spheres in the bouquet is necessarily $\lambda^0-1$. 
 
 Therefore,  
 $$\widetilde H^{n-1}(F_{f, \mathbf 0};\ \Z)=0\hskip 0.2in\textnormal{ and }\hskip 0.2in \widetilde H^{n}(F_{f, \mathbf 0};\ \Z)\cong \Z^{\lambda^0-1}.$$
 Note that \propref{prop:equal} implies $\lambda^0\neq 1$.
 
 \smallskip
 
In terms of the results discussed in the prior section, the above result on the cohomology of $F_{f, \0}$ follows at once from \thmref{thm:bound}.\qed

\bigskip

\noindent {\bf Proof of \propref{prop:0=1}.} Suppose $\lambda^0=1$ and $\lambda^1\geq 2$.

\smallskip

In this case, \thmref{thm:bound} immediately yields 
$$\widetilde H^{n-1}(F_{f, \mathbf 0};\ \Z)\cong\Z^{\lambda^1-1}\hskip 0.2in\textnormal{ and }\hskip 0.2in\widetilde H^{n}(F_{f, \mathbf 0};\ \Z)=0.$$ But there is a great deal more that we can say.

We have
$$
1=\lambda^0=\left(\Gamma^1_{f, \mbf z}\cdot V\left(\frac{\partial f}{\partial z_0}\right)\right)_\0\geq\big(\mult_\0 \Gamma^1_{f, \mbf z}\big)\left(\mult_\0 \left(\frac{\partial f}{\partial z_0}\right)\right),
$$
by Corollary 12.4 of \cite{fulton}. Therefore,
$$
\mult_\0 \Gamma^1_{f, \mbf z}=1, \ \  \mult_\0 \left(\frac{\partial f}{\partial z_0}\right)=1, \  \textnormal{ and so }  \  \mult_\0 f=2.
$$

Furthermore, as $\lambda^0=1$, the trace $\operatorname{tr} \alpha_0$ from \thmref{thm:lemod} must be $\pm1$, and so ${\operatorname{mult}}_\0|\Sigma f|=2$. There are two cases: (1) \ $\Sigma f$ has a single irreducible component $C$ of multiplicity 2 at the origin or (2) \ $\Sigma f$ has two smooth irreducible components which are transversely intersected by $V(z_0)$ at the origin.

In case (1), $\Lambda^1=\mu^\circ C$ and $\lambda^1=(\mult_\0 C)\mu^\circ_C=2\mu^\circ_C$. But then, by \propref{prop:inj}, $\tilde b_{n-1}=2\mu^\circ_C-1\leq \mu^\circ_C$. So we conclude that $\mu^\circ_C=1$; hence $\lambda^1=2$, and  $\widetilde H^{n-1}(F_{f, \mathbf 0};\ \Z)\cong\Z$. Furthermore, by \propref{prop:inj}, the internal monodromy of the vanishing cycles around $C-\{\0\}$ must be the identity. \qed
 
\bigskip

\noindent {\bf Proof of \propref{prop:1=2}.} Suppose $\lambda^1=2$ and $\lambda^0\geq 2$.

\smallskip

In this case, \thmref{thm:bound} tells us that, for all primes $\mfp$, $\tilde b_{n-1}+\,\tau_\mfp$ equals $0$ or $1$; this of course allows for the sum to be $0$ for some prime(s) $\mfp_1$ but equal to $1$ for other prime(s) $\mfp_2$. There are two cases:

\smallskip

\noindent (1) Suppose that $\tilde b_{n-1}=0$. Then, for all primes $\mbf p$, $\tau_\mfp$ equals $0$ or $1$. Thus, 
$$\widetilde H^{n-1}(F_{f, \mathbf 0};\ \Z)=0\hskip 0.2in\textnormal{ and }\hskip 0.2in \widetilde H^{n}(F_{f, \mathbf 0};\ \Z)\cong \Z^{\lambda^0-2}\oplus \big(\Z/b\Z\big)$$  for some $b\geq 1$ (here, $b$=1 or $b$ is a product of powers of primes $\mfp$ such that $\tau_\mfp=1$). Note that \propref{prop:equal} implies $\lambda^0\neq 2$.

\bigskip

\noindent (2) Suppose that $\tilde b_{n-1}\neq 0$. Then $\tilde b_{n-1}=1$ and for all primes $\mfp$, $\tau_\mfp=0$. Thus, 
$$\widetilde H^{n-1}(F_{f, \mathbf 0};\ \Z)\cong \Z\hskip 0.2in\textnormal{ and }\hskip 0.2in\widetilde H^{n}(F_{f, \mathbf 0};\ \Z)\cong \Z^{\lambda^0-1}.$$
Again \propref{prop:equal} implies $\lambda^0\neq 2$.
\bigskip

However, there are 3 other cases which we can consider, based on the nature of $|\Sigma f|$. (a) \ $|\Sigma f|$ consists of two smooth irreducible components at the origin, both with $\mu^\circ =1$. \ (b) \ $|\Sigma f|$ has a single irreducible component $C$ of multiplicity 2 at the origin with $\mu^\circ_C=1$,  \ and (c) \ $|\Sigma f|$ has a single smooth component $C$ at the origin such that $\mu^\circ_C=2$. It is this final case where we have more to say.

In Case (c), $m=\mult_\0|\Sigma f|=1$ and so, by \thmref{thm:lemod}, $\operatorname{tr} \alpha_1=\pm 1$. However, the characteristic polynomial  ${\operatorname{char}}_{\alpha_1}(t)$ of $\alpha_1$ has degree 2 and is product of cyclotomic polynomials. It follows that ${\operatorname{char}}_{\alpha_1}(t)=t^2\pm t+1$, which are both irreducible over $\Z$ (or $\Q$). It follows that we cannot be in Case (2) from above, for if $\widetilde H^{n-1}(F_{f, \mathbf 0};\ \Z)\cong \Z$, then the characteristic polynomial of the Milnor monodromy action on $\widetilde H^{n-1}(F_{f, \mathbf 0};\ \Z)$ would have to divide ${\operatorname{char}}_{\alpha_1}(t)$.

\medskip

Therefore, if $\lambda^1=2$,  $\lambda^0\geq 2$, and $m=1$, we conclude that we are in Case (2), $\lambda^0\geq 3$,  $$\widetilde H^{n-1}(F_{f, \mathbf 0};\ \Z)=0\hskip 0.2in\textnormal{ and }\hskip 0.2in \widetilde H^{n}(F_{f, \mathbf 0};\ \Z)\cong \Z^{\lambda^0-2}\oplus \big(\Z/b\Z\big)$$  for some $b\geq 1$. However, there is more that we say about the prime powers that divide $b$.

For each prime $\mfp$, the question is: does ${\operatorname{char}}_{\alpha_1}(t)=t^2\pm t+1$ factor in $(\Z/\mfp\Z)[t]$? A little number theory tells one that if ${\operatorname{char}}_{\alpha_1}(t)=t^2\pm t+1$ factors in $(\Z/\mfp\Z)[t]$, then $\mfp$ is either $3$ or congruent to $1$ modulo $6$. Thus in the isomorphism $\widetilde H^{n}(F_{f, \mathbf 0};\ \Z)\cong \Z^{\lambda^0-2}\oplus \big(\Z/b\Z\big)$, either $b=1$ or $b$ is a product of prime powers where the primes are either $3$ or congruent to $1$ modulo $6$. \qed

\bigskip

\noindent {\bf Proof of \propref{prop:0=2}.} Suppose $\lambda^0=2$ and $\lambda^1\geq 3$.

\smallskip

Unsurprisingly, the analysis here is analogous to the $\lambda^1=2$ situation above. In our current case, \thmref{thm:bound} tells us that  all primes $\mfp$, $\tilde b_{n}+\,\tau_\mfp$ equals 0 or 1. There are two cases.

\smallskip

\noindent (1) Suppose that $\tilde b_{n}=0$. Then, for all primes $\mbf p$, $\tau_\mfp$ equals $0$ or $1$. Thus, 
$$\widetilde H^{n-1}(F_{f, \mathbf 0};\ \Z)\cong\Z^{\lambda^1-2}\hskip 0.2in\textnormal{ and }\hskip 0.2in \widetilde H^{n}(F_{f, \mathbf 0};\ \Z)\cong \Z/m\Z$$  for some $m\geq 1$.

\bigskip

\noindent (2) Suppose that $\tilde b_{n}\neq 0$. Then $\tilde b_{n}=1$ and for all primes $\mfp$, $\tau_\mfp=0$. Thus, 
$$\widetilde H^{n-1}(F_{f, \mathbf 0};\ \Z)\cong \Z^{\lambda^1-1}\hskip 0.2in\textnormal{ and }\hskip 0.2in\widetilde H^{n}(F_{f, \mathbf 0};\ \Z)\cong \Z.$$

\bigskip

Now assume that $m=\mult_\0\Sigma f=2$. Then, analogous to what we saw in the $\lambda^1=2$ situation, \thmref{thm:lemod} implies that $\operatorname{tr} \alpha_0=\pm 1$ and so ${\operatorname{char}}_{\alpha_0}(t)=t^2\pm t+1$. Thus we must be in Case (1) above and either $m=1$ or $m$ is a product of prime powers where the primes are either $3$ or congruent to $1$ modulo $6$. \qed

\bigskip

\noindent {\bf Proof of \propref{prop:equal}.} Suppose $\lambda^0=\lambda^1= 3$.

\smallskip

\thmref{thm:bound} tells us that $\tilde b_{n-1}<3$. \propref{prop:equal} implies $\tilde b_{n-1}\neq 0, 1$. Thus, we must have $\tilde b_{n-1}=\tilde b_n=2$. But then \thmref{thm:bound} also tells us that for all primes $\mfp$, $\tau_\mfp=0$, i.e., there can be no torsion. Thus
$$
\widetilde H^{n-1}(F_{f, \mathbf 0};\ \Z)\cong \Z^2\hskip 0.2in\textnormal{ and }\hskip 0.2in \widetilde H^{n}(F_{f, \mathbf 0};\ \Z)\cong \Z^2.
$$

\smallskip

 Now suppose that $|\Sigma f|$ has exactly one irreducible component $C$. Then $\mult_\0 C\neq 3$, for if it were, then $\mu^\circ_C$ would have to be 1, and \propref{prop:inj} would tell us that $\tilde b_{n-1}\leq 1$; a contradiction. It is also true that $\mult_\0 C\neq 2$, for if it were $\lambda^1$ would be even; a contradiction. Therefore  $C$ is smooth and transversely intersected by $V(z_0)$ at $\0$, and $\mu^\circ_C$ must be 3. \qed

\medskip

\section{Examples or the lack thereof}

\medskip

Here, we wish to give examples of the various cases from the introduction or say that no such examples are known (to us).

\medskip

\begin{exm} Consider the case where $\lambda^0=0$. Examples of this are trivial to produce, simply take a product of an isolated singularity with a complex line, e.g., $f(z_0, z_1, z_2)=z_1^2+z_2^2$. 
\end{exm}

\medskip

\begin{exm} Consider the case where $\lambda^1=1$ and $\lambda^0\geq 1$. It is easy to produce examples of this case, for instance $f=z_2^2-z_1^3-z_0z_1^2$.
\end{exm}

\medskip

\begin{rem} Consider the case where $\lambda^0=1$ and $\lambda^1\geq 2$. {\bf We do not know if this case can actually occur.} Certainly it cannot occur for a smooth critical locus, since we showed $m=\mult_\0|\Sigma f|$ must be 2.
\end{rem}

\medskip

\begin{exm} Consider the case where $\lambda^1=2$ and $\lambda^0\geq 2$. We showed that $\lambda^0$ must be at least 3 in this case. It is easy to produce examples in all 3 subcases (a), (b), and (c) from the proof of \propref{prop:1=2}. We have the examples, in order,
$$f=(z_0^2-z_1^2+z_2^2)z_2, \hskip 0.2in g=(z_1^2-z_0^3+z_2)z_2,  \ \ \textnormal { and }\hskip 0.1in h=z_2^2-z_1^4-z_0z_1^3.$$
Note that, as $h$ is the ``suspension'' of  $z_1^4-z_0z_1^3$, there is no torsion in the cohomology of the Milnor fiber of $h$ even though \propref{prop:1=2} allows for it to exist.

\end{exm}

\medskip

\begin{exm} Consider the case where $\lambda^0=2$ and $\lambda^1\geq 3$. Examples of this are not so easy to produce, but they exist. For ease, let us denote our coordinates by $(x, y, z)$ in place of $(z_0, z_1, z_2)$.

Let $f = (z^2 - x^2-y^2)(z-x)$.  Then, we have $\Sigma f = V(y,z-x)$, and 
$$
V \left ( \frac{\partial f}{\partial z} \right ) = V(2z(z-x) + (z^2-x^2-y^2) ).
$$
We find, as cycles, that
$$
V \left ( \frac{\partial f}{\partial y}, \frac{\partial f}{\partial z} \right ) = V(y,3z+x) + 3 V(y,z-x),
$$
so that $\Gamma^1_{f, \mbf z} = V(y,3z+x)$ and $\Lambda^1_{f,\mbf z} = 3 V(y,z-x)$. It is now easy to show that $\lambda^0=2$ and $\lambda^1=3$.

\smallskip

In fact, we can show that we are in Case (1) of \propref{prop:0=2}. Up to analytic isomorphism, $f$ is the homogeneous polynomial $f = (zx-y^2)z$.  Consequently, $F_{f,\0}$ is diffeomorphic to $f^{-1}(1)$, i.e., to the space where $(zx-y^2)z=1$; thus $z\neq 0$ and $x=(y^2+1/z)/z$.  Thus, $F_{f,\0}$ is homotopy equivalent to $S^1$.
\end{exm}

 \medskip
 
 \begin{rem} Finally, consider the case where $\lambda^0=\lambda^1= 3$. {\bf Once again, we do not know if this case can actually occur.} \end{rem}

\medskip
 
 \section{Remarks}
 
 \medskip

\begin{rem}
Given the result of Siersma in \cite{siersmaisoline}, we suspected that \thmref{thm:bound} should be the correct generalization on the level of cohomology. Our original attack on proving such a generalization was to show that a non-empty relative polar curve necessarily implies that the internal monodromy in \propref{prop:inj} cannot be the identity. However, the example/counterexample of $f=z_2^2-z_1^3-z_0^2z_1^2$ shows that this is not true.
\end{rem}

\bigskip

\begin{rem} At the end of Example 5.1 of \cite{lemodtrace}, we address the case where $\lambda^1=2$,  $\lambda^0\neq 0$, and $\mult_\0|\Sigma f|=1$. There we concluded quickly that there can be no torsion in the cohomology of the Milnor fiber. While it is possible that this is true, certainly we did not give a careful argument that proved it. As of now, the best we can say is what we wrote in \propref{prop:1=2}.
\end{rem}

\bigskip

\begin{rem}
 Throughout this article, we have allowed for torsion in the cohomology of the Milnor fiber. One might wonder if this can actually occur when $\dim_\0\Sigma f=1$.
 
In fact there is a relatively recent example of such a phenomenon given in \cite{yoshinaga}, which is discussed in a very illuminating way by Suciu in 11.3 of \cite{suciuonyoshinaga}.  This example consists of 16 (reduced) hyperplanes through the origin in $\C^3$, and the critical locus consists of 45 lines. It is shown that there is $2$-torsion in the degree 1 homology group of the Milnor fiber, i.e., in the degree 2 cohomology group. Our results in this paper for functions with low L\^e numbers say essentially nothing about this example (though one can use the formulas in Proposition 1.4 of \cite{pervdual} to show that $\lambda^1=165$ and $\lambda^0=900$).

\smallskip
 
However we do not know if such torsion can possibly exist when $\Sigma f$ is a smooth curve or even has a single irreducible component.

  \end{rem}
 
 \bigskip
 
\bibliographystyle{plain}

\bibliography{Masseybib}

\end{document}